\documentclass[a4paper,11pt]{article}

\usepackage{amssymb, amsmath, color}

\def\A{\mathbb{A}}
\def\C{\mathbb{C}}
\def\R{\mathbb{R}}
\def\N{\mathbb{N}}
\def\P{\mathbb{P}}
\def\Q{\mathbb{Q}}

\def \K {\mathbf{K}}

\def\S{\mathcal{S}}

\def\cI {{\cal I}}

\def\dist {{\rm dist}}

\def\sign {{\rm sign}}

\newtheorem{defn}{Definition}

\newtheorem{lemma}[defn]{Lemma}
\newtheorem{proposition}[defn]{Proposition}
\newtheorem{theorem}[defn]{Theorem}
\newtheorem{corollary}[defn]{Corollary}

           {\vspace{3.3mm}
           \noindent{\bf #1}\it}%
           {\vspace{3.3mm}}

\newenvironment{proof}[1]{
  \trivlist \item[\hskip \labelsep{\it #1}]}{\hfill\mbox{$\square$}
  \endtrivlist}

\title{A probabilistic symbolic algorithm to find the
minimum of a polynomial function on a basic closed semialgebraic set}
\author{Gabriela Jeronimo$^{\diamondsuit,\natural,}$\footnote{Partially supported by the following grants: PIP 099/11 CONICET and UBACYT 20020090100069 (2010/2012).}, Daniel Perrucci$^{{\diamondsuit}, \flat,*}$,
\\[3mm]
{\small ${\diamondsuit}$ Departamento de Matem\'atica, FCEN, Universidad de Buenos Aires, Argentina}\\ {\small ${\natural}$ IMAS, CONICET--UBA, Argentina}\\
{\small $\flat$ CONICET, Argentina}
}

\begin{document}

\maketitle

\begin{abstract}
We consider the problem of computing the minimum of a polynomial function $g$ on a basic closed semialgebraic set $E\subset \R^n$.
We present
a probabilistic symbolic  algorithm to find
a finite set of sample points of the subset $E^{\min}$ of $E$ where the minimum of $g$
is attained, provided that $E^{\min}$ is non-empty and has at least one compact connected component.
\end{abstract}

\section{Introduction}
\label{sect_the_problem}

The minimization of polynomial functions over $\R^n$, unrestricted or subject to polynomial constraints, is a classical problem with a variety of applications. In the last years, it has been extensively studied in the algorithmic framework through numerical or symbolic-numerical methods based on certificates of positivity (see, for instance, \cite{Las01}, \cite{Par03}, \cite{PS03}, \cite{NDS06}, \cite{Schw06}, \cite{GSZ10}, \cite{GGSZ12}).

{}From the symbolic computation perspective, a possible way to tackle the problem is to restate
it as a quantifier elimination problem over the reals and to apply a symbolic algorithm to solve this more general task. However, even the most efficient algorithms for quantifier elimination (see \cite[Chapter 14]{BPR06})
lead to high complexities since they are not particulary designed for this specific problem. In
\cite{Saf08}, a probabilistic algorithm for global optimization of polynomial functions over $\R^n$ with good complexity estimates is presented. The better complexity is due to an alternative strategy relying on the computation of generalized critical values. Also, the problem of deciding algorithmically whether the global infimum of a polynomial function is attained is considered in \cite{GS11}.

Recently, in \cite{JePeTs} (see also \cite{EMT10}, \cite{JP10}), another approach based on computer algebra techniques led to a lower bound for the minimum of a polynomial function on a basic closed semialgebraic set, provided that the set where this minimum is attained is compact.
Here, we adapt the techniques underlying these theoretical result to the algorithmic framework.

More precisely, we consider the following problem. Let $\K \subset \R$ be an effective field. Let $f_1, \dots, f_{m}, g \in
\K[x_1, \dots, x_n]$, with $n\ge 2$, and
$$E = \{x \in \R^n \ | \ f_1(x) = \dots = f_{l}(x) = 0, f_{l + 1}(x) \ge  0, \dots,
f_{m}(x) \ge  0\},$$ and suppose that $g$ attains a minimum value
 $g_{\min}$ at $E$.
We look for a symbolic algorithm to compute at least one point in
$$
E^{\min} = \{x \in E \, | \, g(x) = g_{\min}\}.
$$

In this paper we assume that $E^{\min}$ has at least one compact connected component. This is the
case, for instance, in many families of instances of known optimization problems (see \cite[Section 4]{JePeTs}).
Our approach consists in finding a finite set of points containing at least one point
in each compact connected component of $E^{\min}$. The natural tool to use when solving this problem is the Lagrange Multiplier's Theorem; nevertheless, a direct application of this result may lead to a degenerate system or a system with infinitely many solutions.
In order to overcome this difficulty, we apply deformation techniques as in \cite{JPS} (see also \cite[Chapter 13]{BPR06}),
which enable us to deal with ``nice'' systems that, in the limit, define finite sets containing the
required minimizing points.
These sets are described by geometric resolutions, which are parametric representations where the parameter ranges over the set of roots of a  univariate polynomial.
Finally, we compare the values that the given polynomial function $g$ takes at the computed points and obtain the Thom encodings characterizing the minimizers.

The main result of the paper is the following:

\begin{theorem}
\label{mainTheorem} Let $E = \{x \in \R^n \ | \ f_1(x) = \dots = f_{l}(x) = 0, f_{l + 1}(x) \ge  0, \dots,
f_{m}(x) \ge  0\}$ be defined by polynomials $f_1,\dots, f_m \in \K[x_1,\dots, x_n]$ with $n\ge 2$ and degrees bounded by an even integer $d$. Let $g\in \K[x_1,\dots, x_n]$ be a polynomial of degree at most $d$ that attains a minimum value
 $g_{\min}$ at $E$ in a non-empty set $E^{\min}$ with at least one compact connected component.
Algorithm \emph{\texttt{FindingMinimum}} (see Section \ref{sec:algorithm}) is a probabilistic procedure that, taking as input the integer $d$ and the polynomials $f_1,\dots, f_m, g$ encoded by a straight-line program of length $L$,
computes a  family
$$
\big\{ \big( (p_i, v_{i,1}, \dots, v_{i,n}), \tau_i \big)   \big\}_{i \in \cI}
$$
where $\cI$ is a finite set and for every $i \in \cI$, $(p_i, v_{i,1}, \dots, v_{i,n})$ is a geometric resolution in $\K[u]$ and
$\tau_i \in \{-1, 0, 1\}^{\deg p_i}$ is the Thom encoding
of a real root $\xi_i$ of $p_i$ 
such that
the set $$\{(v_{i,1}(\xi_i), \dots, v_{i,n}(\xi_i)) \}_{i \in \cI}$$
is included in $E^{\min}$
and intersects all its compact connected components.
The complexity of the algorithm is $$O\Big((
n^3 (L +  dn + n^{\Omega-1}) D^2\log^2(
D)\log\log^2(D)+ (m+D)D^2\log^3(D))\Upsilon \Big),$$
where
\begin{itemize}
\item $D=\max\limits_{0\le s\le \min\{n,m\}} \binom{n}{s}d^s (d-1)^{n-s}$,
\item $\Upsilon= \sum\limits_{0\le s \le \min\{n, m\}} \sum\limits_{s_1+s_2=s\atop 0\le s_1\le l,\ 0\le s_2\le m-l} \binom{l}{s_1} \binom{m-l}{s_2} 2^{s_1}\le \sum\limits_{0\le s \le \min\{n,m\}}\binom{m}{s} 2^s.$
\end{itemize}
\end{theorem}

In the above statement, $\Omega$  denotes a positive real number such that
for any ring $R$, addition, multiplication and the computation of
determinant and adjoint of matrices in $R^{k \times k}$ can be
performed within $O(k^\Omega)$ operations in $R$. We may assume
$\Omega \le 4$ (see \cite{Berk}) and, in order to simplify
complexity estimations, we also assume that $\Omega \ge 3$.

An extended abstract of this work containing Theorem \ref{mainTheorem} has been accepted
for presentation at the conference Effective Methods in Algebraic Geometry (MEGA) 2013 (\cite{JP12}).

The paper is organized as follows. In Section \ref{sec:notation}, we introduce the basic notation  and state some previous results we  use throughout the paper. In Section \ref{sec:deformation}, we  present the deformation we apply and we prove some of its geometric properties. Section \ref{sec:geomres} is devoted to showing how the deformation leads to a geometric resolution of the finite set we look for. Finally, in Section \ref{sec:algorithm} we show the algorithmic counterparts of the previous theoretical results, proving Theorem \ref{mainTheorem}.

\section{Notation and preliminaries}
\label{sec:notation}

Throughout the paper, we denote $\N$ the set of positive integers. For a field $K$, we write $\overline{K}$ for an algebraic closure of $K$, $K(t)$ for the field of rational functions in a single variable $t$ and $K[[t]]$ for the set of formal power series in $t$.

For a given $n\in \N$, the $n$-dimensional affine and projective spaces over an algebraically closed field $K$ are denoted by  $\mathbb{A}_{K}^n$ and $\mathbb{P}_{K}^n$ respectively. In the case when $K=\C$, we write simply $\A^n$ and $\P^n$.

\subsection{Algorithms and complexity}\label{algycomp}

The algorithms we consider in this paper are described over an effective field $\K\subset \R$. The notion of {\em complexity} of an algorithm we
consider is the number of operations and comparisons over
$\K$ that the execution of the algorithm requires. In this definition of complexity, accessing, reading and writing pre-computed objects is cost free.

The objects we deal with are polynomials with coefficients in
$\K$. In our algorithms we represent each polynomial
either as the array of all its coefficients in a pre-fixed order of
its monomials ({\it dense form}) or by a {\it straight-line
program}. Roughly speaking, a straight-line program (or slp, for
short) over $\K$ encoding a list of polynomials in
$\K[x_1,\dots, x_n]$ is a program without branches (an
arithmetic circuit) which enables us to evaluate these polynomials
at any given point in $\K^n$. The number of instructions in
the program is called the {\em length} of the slp (for a precise
definition  we refer to \cite[Definition 4.2]{Burgisser}; see also
\cite{HS82}). From the dense encoding of a family of $m$ polynomials in $\K[x_1, \dots, x_n]$ of degrees bounded by $d$, we can obtain an slp of length $O\big(m\binom{d+n}{n}\big)$ encoding them. Unless otherwise
stated, throughout the paper, univariate polynomials will be encoded in dense form.

To estimate complexities we will use the following results.
Operations between univariate polynomials with coefficients in a
field $\K$ of degree bounded by $d$ in dense form can be done using
$O(d \log (d) \log \log (d))$ operations in $\K$ (see \cite[Chapters
8 and 9]{vzG}) and gcd or resultant computations by means of the Extended Euclidean Algorithm (see \cite[Chapter 11]{vzG}) can be performed within $O(d \log^2(d)\log\log(d))$ operations in $\K$.
Given an slp of length $L$ encoding a univariate polynomial of degree at most $d$, we can obtain
its dense form within $O(dL+ d\log^2(d) \log\log(d))$ operations in $\K$ (see \cite[Corollary 10.12]{vzG}).

From an slp of length $L$ encoding a polynomial $f \in
\K[x_1, \dots, x_n]$, we can compute an slp of length $O(L)$
encoding $f$ and all its first order partial derivatives (see
\cite{BS83}).

\subsection{Geometric resolutions}\label{geometricresolutions}

A way of representing zero-dimensional affine varieties which is
widely used in computer algebra nowadays is  a \emph{geometric
resolution} (see, for instance, \cite{GLS01}).
The precise definition we are going to use is the following:

Let $K$ be a field of characteristic $0$ and $V = \{ z_1,
\dots,z_D \} \subset \A_{\overline K}^n$ be a zero-dimensional
variety defined by polynomials in $K[x_1, \dots, x_n]$. Given a
\emph{separating} linear form $\ell = \alpha_1 x_1 + \dots + \alpha_n x_n\in
K[x_1, \dots, x_n]$ for $V$ (that is, a linear form $\ell$ such that
$\ell (z_i)\ne \ell(z_j)$ if $i \ne j$), the following
polynomials completely characterize the variety $V$:
\begin{itemize}
\item the \emph{minimal polynomial} $p:= \prod_{1 \le i \le D} (u
- \ell(z_i))\in K[u]$ of $\ell$ over the variety $V$ (where
$u$ is a new variable),
\item polynomials
$v_1,\dots, v_n \in K[u]$ with $\deg( v_j )< D$ for every $1\le j
\le n$ satisfying
$$ V = \big\{ \big(v_1(\xi) ,\dots, v_n (\xi) \big) \in
\overline{K}^n \mid \xi \in \overline{K} ,\ p(\xi) = 0 \big\}.$$
\end{itemize}
The family of univariate polynomials $(p,  v_1,\dots, v_n)$
is called a \emph{geometric resolution} of $V$ (associated
with the linear form $\ell$).
Note that if $K$ is a subfield of $\R$, the real roots of $p$ correspond to the real points of the variety $V$.

Given geometric resolutions $(p_1, v_{11},\dots, v_{1n})$ and $(p_2, v_{21}, \dots, v_{2n})$  of two zero-dimensional varieties $V_1$ and $V_2$ in $\A^n$ consisting of $D_1$ and $D_2$ points respectively, associated with the same linear form $\ell$ which separates the points in $V_1 \cup V_2$, we can obtain a geometric resolution of $V_1\cup V_2$ within complexity $O(n D \log^2(D) \log\log(D))$, where $D=\max\{D_1,D_2\}$, by means of the Chinese Remainder Theorem using the Extended Euclidean Algorithm.

\subsection{
Thom encoding of real algebraic numbers}
\label{subsec:Thom}

The \emph{Thom encoding} of real algebraic numbers provides an algebraic approach to
distinguish the different real roots of a real univariate polynomial.
We recall here its definition and  main properties (see \cite[Chapter 2]{BPR06}).

Given $p \in \K[u]$ and a real root $\xi$ of $p$, the Thom encoding of
$\xi$ as a root of $p$ is the sequence $(\sign(p'(\xi)), \dots, \sign(p^{(\deg p)}(\xi)))$,
where we represent the sign with an element of the set $\{-1, 0, 1 \}$. If the sign of the
leading coefficient of $p$ is known, the Thom encoding can be shortened to
$(\sign(p'(\xi)), \dots, \sign(p^{(\deg p-1)}(\xi)))$.

Two different real roots of $p$ have different Thom encodings. In addition,
given the Thom encodings $(\sigma_{1,1}, \dots, \sigma_{1,\deg p})$ and
$(\sigma_{2,1}, \dots, \sigma_{2,\deg p})$ of two different real roots $\xi_1$ and $\xi_2$ of $p$,
it is possible to decide which is the smallest between $\xi_1$ and $\xi_2$ as follows:
Consider the largest value of $k$ such that $\sigma_{1,k} \ne \sigma_{2,k}$; then $k < \deg p$, since
$p^{(\deg p)}$ is a constant. Also $\sigma_{1,k+1} = \sigma_{2,k+1} \ne 0$, since otherwise
$\xi_1$ and $\xi_2$ would have the same Thom encoding with respect to the polynomial $p^{(k+1)}$
and therefore $\xi_1 = \xi_2$. Then,
\begin{itemize}
\item if $\sigma_{1, k+1} = \sigma_{2, k+1} = 1$, we have that $\xi_1 < \xi_2$ if and only if $
\sigma_{1, k} < \sigma_{2, k}$,
\item if $\sigma_{1, k+1} = \sigma_{2, k+1} = -1$, we have that $\xi_1 < \xi_2$ if and only if $
\sigma_{1, k} > \sigma_{2, k}$.
\end{itemize}

\section{The deformation}
\label{sec:deformation}

\subsection{Defining the deformation}

Here we introduce  the deformation we use. We denote:
\begin{itemize}

\item $q_0:=n+1$ and $q_1<\dots< q_m$ the first $m$ prime numbers greater than $n+1$. Let
$A \in \Q^{(m+1)\times(n+1)}, A = (a_{ij})_{0 \le i \le m, \ 0 \le j \le n}$ be the Cauchy
matrix defined by $a_{ij} = \frac{1}{q_i-j}$ (note that each submatrix of $A$ has maximal rank and $a_{ij} > 0$ for every $i,j$).

\item For $e \in \N, T_e $ the Tchebychev polynomial of degree $e$ (see \cite[Section 6.1]{KC}).

\item $\tilde g(x) = \sum_{1\le j\le n} a_{0j} T_d(x_j)$ and, for every $1 \le i \le m$, $\tilde f_i(x) = a_{i0}+ \sum_{1\le j\le n} a_{ij}\big(T_d(x_j)+1\big) $.

\item $G(t,x) = tg(x) + (1-t)\tilde g(x)$ and, for every $1 \le i \le m$,
$F_i^+(t,x) = tf_i(x) + (1-t)\tilde f_i(x)$ and
$F_i^-(t,x) = tf_i(x) - (1-t)\tilde f_i(x)$.

\item For every $S \subset \{1, \dots, m\}$ and
$\sigma \in \{+,-\}^{S}$,
$$
\hat V_{S, \sigma} = \{ \ (t,x,\lambda) \in \A \times\A^n \times \P^{|S|} \ | \
F_i^{\sigma_i}(t,x) = 0 \hbox{ for every } i \in S,
$$
$$
\lambda_0\nabla_{x} G(t,x) = \sum_{i \in S} \lambda_i \nabla_{x}F_i^{\sigma_i}(t,x)
\ \}.
$$

We consider the decomposition of  $\hat V_{S, \sigma}$ as $\hat V_{S, \sigma} =
V_{S, \sigma}^{(t)} \cup V_{S, \sigma}$, where
\begin{itemize}
\item $V_{S, \sigma}^{(t)}$ is the union of the irreducible components of $\hat V_{S, \sigma}$ included in ${t = t_0}$ for some $t_0 \in \C$,
\item $V_{S, \sigma}$ is the union of the remaining irreducible components of
$\hat V_{S, \sigma}$.
\end{itemize}

\item For a group of variables $y$, $\Pi_y$ the projection to the coordinates $y$.
\end{itemize}

\subsection{Geometric properties}\label{sec_def_propoerties}

Let $C$ be a compact connected component of $E^{\min}$.
For $\delta > 0$, we write
\begin{itemize}
\item $C_{= \delta} = \{x \in \R^n \ | \ \dist(x, C) = \delta\},$
\item $C_{\le \delta} = \{x \in \R^n \ | \ \dist(x, C) \le \delta\},$
\item $C_{< \delta} = \{x \in \R^n \ | \ \dist(x, C) < \delta\}.$
\end{itemize}

Let $\mu > 0$ such that $C_{\le \mu}$ and $E_{\min} \setminus C$ do not intersect.

We consider
\begin{displaymath}
  \begin{aligned}
    \tilde E = \{(t,x) \in \R \times \R^n \ | \ & F^+_1(t,x) \ge 0,  \dots,  F^+_{l}(t,x) \ge 0, & F^+_{l + 1}(t,x) \ge  0, \dots, F^+_{m}(t,x) \ge  0, &\\
    &  F^-_1(t,x) \le 0 , \dots,  F^-_{l}(t,x) \le 0  \}
  \end{aligned}
\end{displaymath}
and,  for every $0 \le t \le  1$,
$$E_{t} = \Pi_x\big(\tilde E \cap \big(\{t\} \times \R^n \big)\big) \subset \R^{n}.$$
Note that $E_0 = \R^n$, $E_1 = E$ and for $0 \le t_1 \le t_2 \le 1$, $E_{t_2} \subset E_{t_1}$.

\begin{lemma}\label{auxlemmacomp}
There exists $0 < \varepsilon < 1$ such that for every $1 - \varepsilon \le t \le 1$,
the minimum value that $G(t, \cdot  )$ takes on $E_t \cap C_{\le \mu }$ is not attained at any point in
$E_t \cap C_{= \mu }$.
\end{lemma}

\begin{proof}{Proof.}
The minimum value of $g(\cdot ) = G(1, \cdot  )$ at $E_1 \cap C_{\le \mu }$ is
$g_{\min}$, which over this set is only attained at points in $C$ and, therefore, not in
$E_1 \cap C_{= \mu }$.
Assume the claim does not hold. Then there exists a strictly increasing sequence of positive numbers
$(t_k)_{k \in \N}$ converging to $1$ such that the minimum value of $G(t_k, \cdot  )$ at $E_{t_k}
\cap C_{\le \mu }$ is attained at
a point $z_k \in E_{t_k} \cap C_{= \mu }$.

The sequence $(z_k)_{k \in \N}$ is contained in the compact set $C_{=\mu}$; therefore,  without loss of generality,
we may assume that it converges to a point $z \in C_{=\mu}$.
On the other hand, the sequence $(t_k, z_k)_{k \in \N}$ is contained in $\tilde E$ and converges to $(1, z)$, then
$(1, z) \in \tilde E$ and $z \in E_1$.

We will prove that $g(z) = g_{\min}$.
Take $z' \in E_1 \cap C_{\le \mu}$, then $z' \in E_{t_k} \cap C_{\le \mu}$ for every $k \in \N$ and
$$g(z') = G(1, z') = \lim_{k \to \infty} G(t_{k}, z') \ge
\lim_{k \to \infty} G(t_{k}, z_k) = G(1,z) = g(z).$$
Then $g$ attains its minimum on $E_1 \cap C_{\le \mu}$ at $z$, which is $g_{\min}$.
This leads to a contradiction since this value is not attained at any point in $E_1 \cap C_{= \mu }$.
\end{proof}

The following proposition shows that in order to obtain minimizers for
the polynomial function $g$ on the compact connected component $C$,
it is enough to consider at most as many of the equations and inequations
defining $E$ as the number of variables.
We define the set
$$\S
= \{(S, \sigma) \ | \
S \subset \{1, \dots, m\} \hbox{ with } 0 \le |S| \le n \hbox{ and }
\sigma \in \{+, -\}^S \hbox{ with } \sigma_i = + \hbox{ for } l + 1 \le i \le m
\}.
$$

\begin{proposition}\label{prop_prin}
Let $C$ be a compact connected component of $E^{\min}$.
There exist $z \in C$ and $(S, \sigma) \in \S$
such that $z \in \Pi_{x}(V_{S,\sigma}\cap \{t = 1\})$.
\end{proposition}

We point out that if $E$ (resp. $E^{\min}$) is compact, we can
easily prove Proposition \ref{prop_prin} adapting the arguments
in the proof of \cite[Proposition 7]{JePeTs} (resp.  \cite[Theorem 14]{JePeTs}). Now we prove it under our weaker assumptions.

\begin{proof}{Proof of Proposition \ref{prop_prin}.}
Consider $0 < \varepsilon < 1$ such that
\begin{itemize}
\item for every $1 - \varepsilon \le t \le 1$, the minimum value that $G(t, \cdot  )$ takes on $E_t \cap C_{\le \mu }$ is not attained at any point in
$E_t \cap C_{= \mu }$,
\item for every $S \subset \{1, \dots, m\}$ and
$\sigma \in \{+, -\}^S$, $\Pi_{t}(V_{S, \sigma}^{(t)}) \cap
(1 - \varepsilon, 1) = \emptyset$,
\item for every $1 - \varepsilon \le t \le 1$,  $S \subset \{1, \dots, m\}$ with $|S| > n$ and
$\sigma \in \{+, -\}^S$, the set
$$\{x \in \A^n \ | \ F_i^{\sigma_i}(t,x) = 0 \hbox{ for every } i \in S  \}$$
is empty.
\end{itemize}

The existence of such an $\varepsilon$ follows from Lemma \ref{auxlemmacomp},
the finitness of $\Pi_{t}(V_{S, \sigma}^{(t)})$ for every $S$ and $\sigma$,  and an
adaptation of the arguments in  \cite[Lemma 21]{JPS} or \cite[Lemma 4]{JePeTs}.

Let $(t_k)_{k \in \N}$ be an increasing sequence
converging to $1$ with $t_1 > 1 - \varepsilon$ and let
$z_k \in E_{t_k} \cap C_{< \mu }$ be a point
such that $G(t_k, \cdot  )$ attains its minimum value on the set $E_{t_k} \cap C_{\le \mu }$ at
$z_k$. Without loss of generality,
we may assume that the sequence $(z_k)_{k \in \N}$ is convergent to a point $z \in E_1 \cap C_{\le \mu}$,
and proceeding as in the proof of Lemma \ref{auxlemmacomp}, we have that $z \in C$.

Now, for every $k \in \N$ and every $x \in \R^n$, at most one  $F_i^{+}(t_k, x)$
and $F_i^{-}(t_k, x)$ may vanish. Let
$$S_k = \{i \in \{1, \dots, l\} \ | \ F^+_i(t_k, z_k) = 0 \hbox { or } F^-_i(t_k, z_k) = 0 \}
\cup \{i \in \{l+1, \dots, m\} \ | \ F^+_i(t_k, z_k) = 0 \};$$
then $0\le |S_k| \le n$.
Without loss of generality, we may assume that $S_k$ is the same set $S$ for every $k \in \N$;
moreover, we may assume that, for each $i \in S \cap \{1, \dots, l\}$, it is always the same, $F^+_i(t_k, z_k)$ or
$F^-_i(t_k, z_k)$, the one which vanishes, thus defining a function $\sigma \in \{+,-\}^S$ with $\sigma_i= +$ for $l+1\le i \le m$. Then, $(S, \sigma) \in \mathcal{S}$.

For $k \in \N$, if the
set
$\{\nabla_{x}F_i^{\sigma_i}(t_k,z_k), i \in S \}$ is linearly independent,
since the function $G(t_k, \cdot)$ attains a local minimum at the point $z_k$ when restricted to the set
$E_{t_k} \cap C_{<\mu}$, by the Lagrange Multiplier's Theorem,
there exists $(\lambda_{i,k})_{i \in S}$ such that
$$\nabla_{x} G(t_k,z_k) = \sum_{i \in S} \lambda_{i,k} \nabla_{x}F_i^{\sigma_i}(t_k,z_k).$$
We take $\lambda_{0,k} = 1$ and we have that $(t_k, z_k, (\lambda_{0,k},(\lambda_{i,k})_{i \in S})) \in \hat V_{S,\sigma}$;
but since $t_k \not \in \Pi_t(V^{(t)}_{S,\sigma})$, we have that $(t_k, z_k, (\lambda_{0,k},(\lambda_{i,k})_{i \in S}))$
$\in V_{S,\sigma}$.
On the other hand, if $\sum_{i \in S} \lambda_{i,k} \nabla_{x}F_i^{\sigma_i}(t_k,z_k)= 0$ with
$(\lambda_{i,k})_{i \in S} \ne 0$,  we take $\lambda_{0,k} = 0$
and, as in the previous case,  $(t_k, z_k, (\lambda_{0,k} ,(\lambda_{i,k})_{i \in S}))
\in V_{S,\sigma}$.

Without loss of generality,
we may assume that $(\lambda_{0,k},(\lambda_{i,k})_{i \in S})_{k \in \N}$ converges to a point
$(\lambda_{0},(\lambda_{i,0})_{i \in S})) \in \P^{|S|}$; then
$(1,z,(\lambda_{0},(\lambda_{i,0})_{i \in S})) \in V_{S,\sigma}$ and, therefore,
$z \in \Pi_{x}(V_{S,\sigma} \cap \{t = 1 \})$ as we wanted to prove.
\end{proof}

\begin{proposition}\label{finitud}
For every $(S,\sigma) \in \mathcal{S}$, we have that
$\Pi_{x}(V_{S, \sigma} \cap
\{t=1\})$ is a finite set.
\end{proposition}

The proof of Proposition \ref{finitud} will follow from arguments in the next section. From Propositions
\ref{prop_prin} and \ref{finitud} we deduce the following:

\begin{corollary}\label{sample_set} Under the assumptions of Theorem \ref{mainTheorem}, the set
$$
\bigcup_{(S, \sigma) \in \S } \Pi_{x}(V_{S, \sigma} \cap
\{t=1\})
$$
is finite and contains a point in every compact connected component of $E^{\min}$.
\end{corollary}

Due to Corollary \ref{sample_set}, to solve the problem we are considering,
in the next section we will focus on describing
the set $\Pi_{x}(V_{S, \sigma}\cap
\{t=1\})$ for every $(S,\sigma)\in \mathcal{S}$.

\section{A geometric resolution}
\label{sec:geomres}

Throughout this section, we consider fixed $(S, \sigma) \in \S$ and denote $\hat V = \hat V_{S, \sigma}$ and $V = V_{S, \sigma}$;
moreover, for simplicity, we suppose $S = \{1, \dots, s\}$ and $\sigma = \{+\}^S$.

For $1\le j \le n$, let
$$g_j(x, \lambda) = \lambda_0 \frac{\partial g}{\partial x_j} - \sum_{1 \le i \le s}
\lambda_i \frac{\partial f_i}{\partial x_j}
\in \K[x_1,\dots, x_n,\lambda_0,\dots, \lambda_s],
$$
$$
\tilde g_j(x, \lambda) = \lambda_0 \frac{\partial \tilde g}{\partial x_j} -
\sum_{1 \le i \le s}
\lambda_i \frac{\partial \tilde  f_i}{\partial x_j}
\in \K[x_1,\dots, x_n,\lambda_0,\dots, \lambda_s],$$
$$
G_j(t, x, \lambda) = tg_j(x, \lambda) + (1-t)\tilde g_j(x, \lambda)
\in
\K[t, x_1,\dots, x_n,\lambda_0,\dots, \lambda_s].$$

These polynomials are
homogeneous of degree $1$ in the
variables $\lambda$; therefore, by the
multihomogeneous B\'ezout theorem (see, for instance, \cite[Chapter 4,
Section 2.1]{Shafarevich}), the degree of
the varieties $\hat V \cap \{t = t_0\}$ for $t_0 \in \C$
is bounded by
\begin{equation*}
D_s:= \binom{n}{s}d^s(d-1)^{n-s}.
\end{equation*}

The next lemma shows the key properties of the initial system in the deformation (c.f. \cite[Lemma 20]{JPS}).

\begin{lemma}\label{desco_init_variety}
The polynomials $\tilde f_1,\dots, \tilde f_s, \tilde g_1, \dots, \tilde g_n$
define a $0$-dimensional variety in $\A^n
\times \{\lambda_0\ne 0\} \subset \A^n \times \P^{s}$ with $D_s$ distinct points
$w_1,\dots, w_{D_s}$ satisfying $\Pi_x(w_i) \ne \Pi_x(w_j)$ for $i\ne
j$, and the Jacobian determinant of $\tilde f_1,\dots, \tilde f_s$ and the polynomials obtained from
$\tilde g_1, \dots, \tilde g_n$ dehomogenizing with $\lambda_0=1$ does not
vanish at any of these points.
\end{lemma}

\begin{proof}{Proof.} Recalling that $\tilde g(x) = \sum_{1\le j\le n} a_{0j} T_d(x_j)$ and  $\tilde f_i = a_{i0}+\sum_{1\le j\le n} a_{ij} (T_d(x_j)+1)$,  we have that, for every $1\le j \le n$,
$$\tilde g_j (x, \lambda) = T_d'(x_j) \Big(a_{0j}\lambda _0-\sum_{1\le i\le s} a_{ij} \lambda _i\Big).$$

Therefore, the solution set of the system $\tilde f_1,\dots, \tilde f_s, \tilde g_1,\dots, \tilde g_n$ can be decomposed as
$$\bigcup_{B\subset \{1,\dots, n\}} \{T'_d(x_j)=0 \ \forall j\in B,\ \tilde f_1(x)=0,\dots, \tilde f_s(x)=0\}\times \{ a_{0j}\lambda _0-\sum_{1\le i\le s} a_{ij} \lambda _i=0 \ \forall j\notin B\}.$$

By our assumption on the matrix $A$, if $|B|= n-s$, the linear system $a_{0j}\lambda _0-\sum_{1\le i\le s} a_{ij} \lambda _i=0 \ \forall j\notin B$ has a unique solution $\Lambda_B \in \P^s$; moreover, this solution lies in $\{\lambda_0 \ne 0\}$.

For a fixed $B\subset \{1,\dots, n\}$ with $|B|= n-s$, taking into account that $T_d'$ has $d-1$ real roots and  $T_d$ takes the value $1$ or $-1$ at each of these roots, we have that $$S_B:=\{T'_d(x_j)=0 \ \forall j\in B,\ \tilde f_1(x)=0,\dots, \tilde f_s(x)=0\}$$ decomposes as the union of the sets
$$S_{B,e}:= \{T_d'(x_j) =0,\ T_d(x_j)=e(j) \ \forall j\in B, \ \tilde f_1^{B, e} =0,\dots, \tilde f_s^{B, e}=0\}$$
for all $e: B \to \{1,-1\}$, where $\tilde f_i^{B,e}\in \mathbf{K}[x_j; j\notin B]$ denotes the polynomial obtained from $\tilde f_i$ by replacing $T_d(x_j) = e(j)$ for every $j\in B$.

Without loss of generality, in order to simplify notation, assume $B= \{s+1,\dots, n\}$. Then, for $e: B\to \{1,-1\}$, the system $\tilde f_1^{B,e} =0,\dots, \tilde f_s^{B,e} =0$ can be written in the form
$$A_B \left(\begin{array}{c} T_d(x_1) +1 \\ \vdots \\ T_d(x_s)+1 \end{array}\right) =
\left( \begin{array}{c} \alpha_1^{B,e} \\ \vdots \\ \alpha_s^{B,e} \end{array}\right)$$
where $A_B :=(a_{ij})_{1\le i, j \le s}$ and $\alpha_{i}^{B,e}= - a_{i0}-\sum_{s+1\le j\le n} a_{ij} (e(j)+1) $ for $1\le i\le s$. Since $A_B$ is invertible, we can solve the underlying linear system for $T_d(x_1)+1, \dots, T_d(x_s)+1$. By applying Cramer's rule, it can be seen that the coordinates of the solution to this linear system are rational numbers where the denominators are a multiple of the prime number $q_s$, whereas the numerators are relatively prime with $q_s$; therefore, no coordinate of a solution is an integer number. We deduce that the above system is equivalent to a system of the form
$$T_d(x_1) = c_1^{B,e},\dots, T_d(x_s)= c_s^{B,e}$$
where $c_i^{B,e}\ne \pm 1$ for every $1\le i \le s$. It follows that each of the equations has $d$ distict roots, none of which equals a root of $T'_d$ (thus, the sets $S_{B,e}$ are mutually disjoint).

 Moreover, the Jacobian matrix of $\tilde f_1,\dots, \tilde f_s$ and the polynomials obtained from
$\tilde g_1, \dots, \tilde g_n$ dehomogenizing with $\lambda_0=1$ evaluated at any of
its solutions is of the form
$$ \begin{array}{cc}
\begin{array}{cc}
          s &\{ \cr s & \{  \cr n-s &\{ \cr
         \end{array} & \left(\begin{array}{c|c|c}
                        \ C_1 & 0 & \ 0 \cr
                        \hline  *  & 0 &  \ C_2  \cr
                        \hline  *  &  \ C_3  \ &  \ *  \cr
                         \end{array}\right) \cr &
         \begin{array}{ccc}  \, \underbrace{}_s  & \underbrace{}_{n-s} \ &  \hspace{-2mm}\underbrace{}_{s}   \cr
                                     \end{array}
         \end{array}.
 $$
It is easy to see that $C_1$, $C_2$ and $C_3$ are invertible matrices and so, the Jacobian determinant does not vanish.

We conclude that $S_B$ consists of $(d-1)^{n-s} d^s$ distinct points in $\A^n$ for every $B$ with $|B|= n-s$.
Hence, the system $\tilde f_1,\dots, \tilde f_s, \tilde g_1,\dots, \tilde g_n$ has $D_s$ isolated solutions in $\A^n \times \P^s$ whose projections to $\A^n$ are all distinct. Since $D_s$ is an upper bound for the degree of the variety the system defines, it follows that these are all its solutions.
\end{proof}

As a consequence of Lemma \ref{desco_init_variety},
 it follows that all the irreducible components of $V$ intersect the set $\{ t=0\}$ and have dimension $1$.
Proposition \ref{finitud} is immediate from this fact.
Moreover, the following further properties of the induced deformation hold.

\begin{lemma}\label{prop_var_t_como_coef} The variety  defined in
$\A^n_{\overline{\K(t)}} \times \P_{\overline{\K(t)}}^{s}$ by
$F_1, \dots, F_s, G_1, \dots, G_n$ is 0-dimensional
and has  $D_s$ distinct  points $W_1,\dots, W_{D_s}$ in $\{\lambda_{0}\ne 0\}$ such
that $\Pi_x(W_i)\ne \Pi_x(W_j)$ for $i\ne j$. Moreover, these points
can be considered as elements in $\K [[t]]^{n+s}$.
\end{lemma}

\begin{proof}{Proof:} The multihomogeneous  B\'ezout Theorem  states that the degree
of the variety is bounded by $D_s$. If $w_1,\dots, w_{D_s}$ are the
common zeros of $\tilde f_1,\dots, \tilde f_s, \tilde g_{1}, \dots, \tilde g_n$, the Jacobian with respect to $x_1,\dots, x_n, \lambda_1,\dots,
\lambda_{s}$
of  $F_1, \dots, F_s,$ and the polynomials obtained from $ G_1, \dots, G_n$ dehomogenizing with $\lambda_0=1$ at $t=0$ and $(x, \lambda) = w_i$ is nonzero. The result
follows applying the Newton-Hensel lifting (see for example
\cite[Lemma 3]{HKPSW}).
\end{proof}

Consider now new variables $y_1, \dots, y_n$ and define $\ell(x,
\lambda, y) = \ell(x, y) = \sum_{1\le j \le n} y_j x_j$. For $\alpha_1,
\dots, \alpha_n \in \C$,  let $\ell_{\alpha}(x, \lambda) =
\ell_{\alpha}(x) = \sum_{1\le j \le n} \alpha_j x_j$. Let
$$
P(t, u, y) = \prod_{1\le i\le D_s} \big(   u -   \ell(W_i, y) \big)  =
\frac{\sum_{0\le h\le D_s} p_h(t, y)u^h}{q(t)}=\frac{\hat P(t, u,
y)}{q(t)}\in \mathbb{K}(t)[u, y],
$$
with $\hat P(t, u, y) \in \K[t, u, y]$ with no factors in
$\K[t]$.
Let $Q(u,y) = \gcd(\hat P(1, u, y), \partial \hat P/\partial u (1, u, y))$.
Then, for generic $\alpha\in \C^n$, if
$$p(u):= \frac{\hat P(1, u, \alpha)}{Q(u, \alpha)}$$
and, for every $1\le j \le n$,
$$v_j(u):=  -\frac{\frac{\partial \hat P}{\partial y_j}(1, u, \alpha)}{Q(u,
\alpha)} \ \left(\frac{\frac{\partial
\hat P}{\partial u}(1, u, \alpha)}{Q(u, \alpha)}\right)^{-1} \mod p(u),$$
we have that
$
\big(p(u),v_1(u),\dots, v_n(u)\big)
$
is the geometric resolution associated to $\ell_\alpha$ of a finite set $\mathcal{P}$ containing $\Pi_x(V \cap \{ t=1\})$ (see, for instance, \cite[Algorithm 9]{GLS01} or \cite[Proposition 8]{JPS}).

The computation of $\hat P$ will be done by means of a Newton-Hensel based approximation. The required precision is obtained from the degree bound $\deg_t \hat P(t,u,y) \le n D_s$, which can be proved as in \cite[Lemma 9]{JPS}.

\section{The algorithm}\label{sec:algorithm}

In this section we present our algorithms and prove the main result of the paper.

{}From Corollary \ref{sample_set}, we know that the finite set
$$
\bigcup_{(S, \sigma) \in \S } \Pi_{x}(V_{S, \sigma} \cap
\{t=1\})
$$
contains a point in every compact connected component of $E^{\min}$; nevertheless,
for a fixed $(S, \sigma) \in \S$, the set
$\Pi_{x}(V_{S, \sigma}\cap
\{t=1\})$ is not necessarily contained in $E^{\min}$, or may even
have an empty intersection with $E$.
The idea of our main algorithm is to compute first finite sets $\mathcal{P}_{S, \sigma}$ containing $\Pi_{x}(V_{S, \sigma} \cap
\{t=1\})$; then,  look for for the points of each $\mathcal{P}_{S,\sigma}$ that lie in $E$ and finally, compare the values that the function $g$ takes at these points.

First, we introduce three auxiliary subroutines we use to construct our main procedure.

Our first subroutine is an algorithm to compute the geometric resolution of the finite set $\mathcal{P}_{S,\sigma}$ containing  $\Pi_x(V_{S,\sigma} \cap \{t = 1\})$
introduced in the previous section. This algorithm relies on the global Newton lifting  from \cite{GLS01} and it is essentially the procedure underlying \cite[Proposition 13]{JPS}; we include it here for the sake of completeness. In order to simplify notation, we assume that $S = \{1, \dots, s\}$ and $\sigma = \{+\}^S$.

\bigskip

\noindent \textbf{Algorithm} \texttt{GeometricResolution}

\bigskip
\noindent INPUT:
Polynomials $f_1,\dots, f_s, g\in \K[x_1,\dots, x_n]$ encoded by an slp of length $L$,
an even integer $d\ge \deg(f_i), \deg(g)$, and a linear form $\ell_\alpha\in \K[x_1,\dots, x_n]$.

\bigskip
\noindent OUTPUT: The geometric resolution $(p, v_{1}, \dots, v_{n})$ associated to $\ell_\alpha$ of a finite set  $\mathcal{P}$ containing $\Pi_x(V \cap \{t = 1\})$.

\begin{enumerate}

\item Compute the geometric resolution
associated to $\ell_\alpha(x)=\alpha_1 x_1 + \cdots
+ \alpha_n x_n$ of the variety defined in $\A^{n+s}$ by the
(dehomogenized) system $\tilde f_1, \dots, \tilde f_s, \tilde g_1, \dots, \tilde g_n$ as follows:

\begin{enumerate}

\item For every $B \subset \{1, \dots, n\}$ and $e:B \to \{-1, 1\}$, compute the
geometric resolution associated to $\ell_\alpha(x)$
of the variety $S_{B, e}$.

\item Compute the
geometric resolution associated to $\ell_\alpha(x)$
of the variety $\bigcup_{B, e} S_{B, e}$.

\end{enumerate}

\item Compute the geometric resolution associated to
$\ell(x,y)=y_1x_1+\cdots+y_nx_n$ of the variety defined by the
(dehomogenized) system $\tilde f_1, \dots, \tilde f_s, \tilde g_1, \dots, \tilde g_n$ over $\overline{\K(y)}$, modulo the ideal $(y_1 -
\alpha_1, \dots, y_n - \alpha_n)^2$.

\item Compute $P(t, u, y) \mod ((t)^{2nD_s
+ 1} + (y_1 - \alpha_1, \dots, y_n - \alpha_n)^2)\K[[t]][u, y]$.

\item Compute
$\hat P(1,u, \alpha) = \sum_{0\le h\le D_s} p_h(1,\alpha) u^h$ and
$\frac{\partial \hat P}{\partial y_j}(1,u, \alpha) = \sum_{0\le h\le D_s}
\frac{\partial p_h}{\partial y_j}(1,\alpha) u^h$ as follows:

\begin{enumerate}

\item Compute $p_h(t, \alpha)$ and $\frac{\partial p_h}{\partial
y_j}(t, \alpha)$ $(1\le j \le n,\, 0\le h \le D_s)$.

\item Evaluate $t=1$.

\end{enumerate}

\item Compute $Q(u, \alpha) =
\textrm{gcd}(\hat P(1, u, \alpha),\frac{\partial \hat P}{\partial
u}(1,u, \alpha))$, $p(u)= \frac{\hat P(1, u, \alpha)}{Q(u, \alpha)}$
and, for every $1\le j \le n$,
$v_j(u):=  -\frac{\frac{\partial \hat P}{\partial y_j}(1, u, \alpha)}{Q(u,
\alpha)} \ \left(\frac{\frac{\partial
\hat P}{\partial u}(1, u, \alpha)}{Q(u, \alpha)}\right)^{-1} \mod p(u)$.

\end{enumerate}

\begin{proposition}\label{complalgoritmopararesgeomtipo2}
Given a generic $\alpha \in \K^n$ and  polynomials
$f_1, \dots, f_{m}$ $\in \K [x_1, \dots, x_n]$
of degree bounded by an even integer $d$ and
encoded by an slp of length $L$,
Algorithm \emph{\texttt{GeometricResolution}} computes
the geometric resolution associated to the linear form $\ell_\alpha (x) = \sum_{1\le j \le n} \alpha_j x_j$ of a finite set $\mathcal{P}_{S,\sigma}$ containing $\Pi_x(V_{S,\sigma} \cap \{
t=1\})$
within complexity $O\big( n^3(L +  dn + n^{\Omega-1}) D_s^2\log^2(
D_s)\log\log^2(D_s)\big)$.
\end{proposition}

\begin{proof}{Proof:}

\noindent\textsc{Step 1}(a). The variety $S_{B, e}$ is defined by a square polynomial system in
separated variables; then, the required computation can be achieved
as in \cite[Section 5.2.1]{JMSW} within complexity  $O( D_{B, e}^2
\log^2( D_{B, e})\log \log(D_{B, e}))$, where $D_{B, e}$ is the
cardinality of $S_{B, e}$.

\noindent\textsc{Step 1}(b). This step can be done  within complexity $O(n D_s \log^3 (D_s) \log \log
(D_s))$  following the procedure in Section \ref{geometricresolutions} and
the strategy described in \cite[Algorithm 10.3]{vzG}.

\medskip

\noindent{\sc Step 2.} This step can be done applying \cite[Algorithm
1]{GLS01} within complexity $O(( dn^3 + n^{\Omega+1})$ $ D_s \log (D_s)
\log \log (D_s))$.

\medskip

\noindent {\sc Step 3.} Since $F_1, \dots, F_s, G_1,\dots,G_n$ can be encoded
by an slp of length $O(L + ( d + s)n)$, a geometric resolution of
the variety they define associated with the linear form $\ell(x, y)$ modulo the
ideal $(t)^{2n D_s + 1} + (y_1 - \alpha_1, \dots, y_n - \alpha_n)^2$
can be obtained from the previously computed geometric resolution by
applying \cite[Algorithm 1]{GLS01} within complexity    $O(n^3(L +  dn
+ n^{\Omega-1})  D_s^2 \log^2 ( D_s) \log \log^2 ( D_s))$.

\medskip

\noindent {\sc Step 4}(a).
By expanding $ P(t, u, y) = \sum_{0\le h\le D_s}  \frac{p_h(t, y)}{q(t)}u^h
\in \K[[t]][u, y]$ into powers of  $u$, $(y_1 - \alpha_1), \dots,
(y_n - \alpha_n)$, we have that the coefficients corresponding to
$u^h$ and $u^h(y_j - \alpha_j)$ $(1\le j \le n,\, 0\le h \le D_s)$ are
$ {p_h(t, \alpha)}/{q(t)}$ and ${\frac{\partial p_h}{\partial
y_j}(t, \alpha)}/{q(t)}$ respectively. As the degrees of the
polynomials involved in these fractions are bounded by $nD_s$, they are uniquely determined by their power
series expansions modulo $(t)^{2nD_s + 1}\K[[t]]$ (see
\cite[Corollary 5.21]{vzG}) that were computed at Step 3.
By using \cite[Corollary 5.24 and Algorithm 11.4]{vzG} and
converting all rational fractions to a common denominator, the
computation is done within complexity $O(n^2D_s^2\log^2(D_s) \log \log(D_s))$.

\medskip

\noindent {\sc Step 5.} This step is achieved by means of the Extended Euclidean algorithm and polynomial divisions with remainder
 within complexity $O(n D_s \log^2(D_s) \log \log (D_s))$.
\end{proof}

 The second subroutine presented here tells us, for a finite set $\mathcal{P}_{S,\sigma}$ given by a geometric resolution, if the set $\mathcal{P}_{S, \sigma} \cap E$ is empty, and, if not, it computes the Thom encoding of all the real roots of the minimal polynomial $p_{S, \sigma}$ corresponding to the points   where the minimum value of $g$ on $\mathcal{P}_{S, \sigma} \cap E$ is attained.

\bigskip
\noindent \textbf{Algorithm} \texttt{MinimumInGeometricResolution}

\medskip
\noindent INPUT: A geometric resolution
$(p_{S, \sigma},v_{S, \sigma, 1},\dots, v_{S, \sigma, n})$  in $\K[u]$ of a finite set $\mathcal{P}_{S,\sigma}$, polynomials
$f_1,\dots, f_m, g\in \K[x_1,\dots, x_n]$ encoded by an slp of length $L$, an integer $0\le l\le m$
and an integer $d\ge \deg(f_i), \deg(g)$.

\medskip
\noindent OUTPUT: A boolean variable ``$ {\rm Empty}$'' with the truth value of the statement
``The set $\mathcal{P}_{S,\sigma}\cap E$ is empty"
and a list of elements $\tau_1, \dots, \tau_k \in \{-1,0,1\}^{\deg p_{S, \sigma}-1}$
with $k = 0$ if $ {\rm Empty} = {\rm True}$,
representing
the Thom encodings of all the real roots of
$p_{S, \sigma}$ corresponding to the points   where the minimum value of $g$ on $\mathcal{P}_{S, \sigma} \cap E$ is attained.

\begin{enumerate}
   \item Compute the list of realizable sign conditions for
$f_1(v_{S, \sigma}(u)), \dots, f_m(v_{S, \sigma}(u))$
 over the real zeros of $p_{S, \sigma}(u)$.
  \item  Determine Empty going through the obtained list of realizable sign conditions.
  \item If Empty $=$ False:
  \begin{enumerate}
\item Compute $h(u):= \text{Res}_{\tilde u} (p_{S, \sigma}(\tilde u), u- g(v_{S, \sigma}(\tilde u)))$.
\item Compute the list of realizable sign conditions for $f_1(v_{S, \sigma}(u)), \dots, f_m(v_{S, \sigma}(u))$,
$p_{S, \sigma}'(u),$ $\dots,
p_{S, \sigma}^{(\deg p_{S, \sigma}-1)}(u),$  $h'(g(v_{S, \sigma}(u))), \dots, h^{(\deg p_{S, \sigma}-1)} (g(v_{S, \sigma}(u)))$
over the real zeros of $p_{S, \sigma}(u)$.
\item Determine $\tau_1, \dots, \tau_k$ going through the obtained list of realizable sign conditions.
  \end{enumerate}
\end{enumerate}

\begin{proposition} Given a geoemetric resolution
$(p_{S, \sigma},v_{S, \sigma, 1},\dots, v_{S, \sigma, n})$
of the set $\mathcal{P}_{S,\sigma}$ and the  polynomials
$f_1,\dots, f_m, g\in \K[x_1,\dots, x_n]$ encoded by an slp of length $L$,
Algorithm \linebreak \emph{\texttt{MinimumInGeometricResolution}}
decides whether $\mathcal{P}_{S,\sigma} \cap E$ is empty or not and computes the Thom encodings
of the real roots of $p_{S, \sigma}$ corresponding to the points  where the minimum value of $g$ on $\mathcal{P}_{S, \sigma} \cap E$ is attained
within complexity
$
O(nD^2_s + LD_s+(m+D_s)D^2_s \log^3(D_s))).
$
\end{proposition}

\begin{proof}{Proof.}

{\sc Step 1.} First we compute the dense encoding of the polynomials
$f_1(v_{S, \sigma}(u)), \dots, f_m(v_{S, \sigma}(u))$,  within complexity $O((nD_s+L)D_s+mD_s \log^2(D_s)\log\log(D_s))$.
Then we apply the sign determination algorithm from \cite[Section 10.3]{BPR06} and \cite{Canny93}, following \cite{Perrucci11}, within complexity $O(m D_s^2 \log^3(D_s))$.

\medskip

\noindent {\sc Step 2.} This step can be done within complexity $O(mD_s)$.

\medskip

\noindent {\sc Step 3}(a). At this step the algorithm computes the monic polynomial $h$ whose roots are
the values of $g$ at the points in $\mathcal{P}_{S,\sigma}$.
First we compute the dense encoding of the polynomial $g(v_{S, \sigma})$ and then, the resultant polynomial by multi-point evaluation and interpolation within complexity $O((nD_s+L) D_s + D_s^2 \log^2(D_s) \log\log(D_s))$.

\smallskip
\noindent {\sc Step 3}(b). We continue the sign determination algorithm
adding the polynomials
$p_{S, \sigma}'(u),$ $\dots,
p_{S, \sigma}^{(\deg p_{S, \sigma}-1)}(u),$  $h'(g(v_{S, \sigma}(u))), \dots, h^{(\deg p_{S, \sigma}-1)} (g(v_{S, \sigma}(u)))$
to what we have already computed at Step 1. To do this, we first obtain the dense encoding of the polynomials involved. This step can be done within complexity
$O((m+D_s) D_s^2 \log^3(D_s))$.

\smallskip

\noindent {\sc Step 3}(c). The list of sign conditions computed at Step 3(b) enables us to know the Thom
encoding of every real root $\xi$ of $p_{S, \sigma}$, and to relate $\xi$ with the Thom encoding
of $g(v_{S,\sigma}(\xi))$ as a root of $h$. This information is enough to compare the different values of
$g(v_{S,\sigma}(\xi))$ (see Section \ref{subsec:Thom}) and, so, we can give the Thom encoding as roots of $p_{S, \sigma}$ of the roots
giving the points where the minimum value of $g$ is attained. This step is done within complexity $O((m+D_s) D_s)$.
\end{proof}

Since we know that the minimum value of $g$ on $E$ is also the minimum value of $g$ on
$$
 \bigcup_{(S, \sigma) \in \S } \Pi_{x}(V_{S, \sigma} \cap
\{t=1\}) \cap E,
$$
the set of sample minimizing points will be obtained by comparing the minimum values
that
$g$ takes on  $\mathcal{P}_{S,\sigma} \cap E$ for all $(S, \sigma) \in \S$.
This task is  achieved by the following subroutine.

\bigskip

\noindent \textbf{Algorithm} \texttt{ComparingMinimums}

\bigskip

\noindent INPUT: Geometric resolutions $(p_{S_1, \sigma_1},v_{S_1, \sigma_1, 1},\dots, v_{S_1, \sigma_1, n})$ and $(p_{S_2, \sigma_2},v_{S_2, \sigma_2, 1},\dots, v_{S_2, \sigma_2, n})$ of finite sets $\mathcal{P}_{S_1,\sigma_1}$ and $\mathcal{P}_{S_2,\sigma_2}$ associated with a linear form $\ell_\alpha$, $g \in \K[x_1, \dots, x_n]$ encoded by an slp of length $L$, an integer $d\ge \deg(g)$,  and Thom encodings $\tau_1 \in \{-1, 0, 1\}^{\deg p_{S_1, \sigma_1}-1}$ and $\tau_2
\in \{-1, 0, 1\}^{\deg p_{S_2, \sigma_2}-1}$ of real roots of $p_{S_1, \sigma_1}$ and $p_{S_2, \sigma_2}$ corresponding to points where the minimum of $g$ on $\mathcal{P}_{S_1,\sigma_1}\cap E$ and $\mathcal{P}_{S_2,\sigma_2}\cap E$
is attained.

\medskip
\noindent OUTPUT: An integer ``${\rm Sign}$'' from the set $\{-1, 0, 1\}$ representing the sign of the minimum
value that $g$ takes on  $\mathcal{P}_{S_1,\sigma_1}\cap E$ minus the minimum
value that $g$ takes on  $\mathcal{P}_{S_2,\sigma_2} \cap E$.

\begin{enumerate}
\item Compute a geometric resolution $(p, v_1, \dots, v_n)$ of the union of the sets described by
$(p_{S_1, \sigma_1},$ $v_{S_1, \sigma_1, 1},\dots, v_{S_1, \sigma_1, n})$ and
$(p_{S_2, \sigma_2},v_{S_2, \sigma_2, 1},\dots, v_{S_2, \sigma_2, n})$.

\item Compute  $h(u):= \text{Res}_{\tilde u} (p(\tilde u), u- g(v(\tilde u)))$.

\item Compute the list of realizable sign conditions for
$p_{S_1, \sigma_1}(u), p_{S_1, \sigma_1}'(u),$ $\dots,
p_{S_1, \sigma_1}^{(\deg p_{S_1, \sigma_1}-1)}(u),$
$p_{S_2, \sigma_2}(u), p_{S_2, \sigma_2}'(u),$ $\dots,
p_{S_2, \sigma_2}^{(\deg p_{S_2, \sigma_2}-1)}(u),$
$h'(g(v(u))), \dots, h^{(\deg p)}(g(v(u)))$
over the real zeros of $p(u)$.
\item Determine ${\rm Sign}$ going through the obtained list of realizable sign conditions.

\end{enumerate}

\begin{proposition}
Given geometric resolutions
$(p_{S_1, \sigma_1},v_{S_1, \sigma_1, 1},\dots, v_{S_1, \sigma_1, n})$
of $\mathcal{P}_{S_1,\sigma_1}$
and \linebreak
$(p_{S_2, \sigma_2},v_{S_2, \sigma_2, 1},\dots, v_{S_2, \sigma_2, n})$
of $\mathcal{P}_{S_2,\sigma_2}$, and the Thom encodings of a real root of
$p_{S_1, \sigma_1}$ and $p_{S_2, \sigma_2}$ corresponding to points where the minimum of $g$ on $\mathcal{P}_{S_1,\sigma_1}\cap E$ and $\mathcal{P}_{S_2,\sigma_2}\cap E$
is attained,
Algorithm \emph{\texttt{ComparingMinimums}}
compares these minimums within complexity $O(n D^2_{s_1,s_2} + L D_{s_1,s_2}+ D^3_{s_1,s_2}\log^3(D_{s_1,s_2}))$ where $D_{s_1, s_2}:=\max \{D_{s_1}, D_{s_2}\}$ with $s_1= |S_1|$ and $s_2=|S_2|$.
\end{proposition}

\begin{proof}{Proof.}
{\sc Step 1.} This step is achieved
within complexity $O(n D_{s_1,s_2} \log^2(D_{s_1,s_2}) \log\log(D_{s_1,s_2}))$ as explained in Section \ref{sec:notation}.

\medskip

\noindent {\sc Steps 2, 3 and 4.} Similar to Algorithm \texttt{MinimumInGeometricResolution} Step 3(a), (b) and (c). The overall complexity of these steps is $O(n D^2_{s_1,s_2} + L D_{s_1,s_2}+ D^3_{s_1,s_2}\log^3(D_{s_1,s_2}))$.
\end{proof}

We give now the main algorithm of the paper.
\bigskip

\noindent \textbf{Algorithm} \texttt{FindingMinimum}

\bigskip
\noindent INPUT: Polynomials $f_1,\dots, f_m, g\in \K[x_1,\dots, x_n]$ encoded
by an slp of length $L$, an integer $0\le l\le m$ and an even integer $d\ge \deg(f_i), \deg(g)$.

\bigskip
\noindent OUTPUT: A  family
$
\big\{ \big( (p_i, v_{i,1}, \dots, v_{i,n}), \tau_i \big)   \big\}_{i \in \cI}
$
where $\cI$ is a finite set and for every $i \in \cI$,
$(p_i, v_{i,1}, \dots, v_{i,n})$ is a geometric resolution in $\K[u]$ and
$\tau_i \in \{-1, 0, 1\}^{\deg p_i}$ is the Thom encoding
of a real root $\xi$ of $p_i$.

\begin{enumerate}

\item Take $\alpha=(\alpha_1,\dots, \alpha_n)\in \K^n$ at random and set $\ell_\alpha:= \alpha_1 x_1 +\cdots + \alpha_n x_n$.

\item
 $\mathcal{S}=\{(S, \sigma) \ | \ S \subset \{1,\dots,m\} \hbox{ with } 0\le |S|\le n
\hbox{ and } \sigma \in \{+, -\}^S \hbox{ with } \sigma_i = + \hbox { for } l  + 1 \le i  \le m
\}
$.

    \item Take $(S_1, \sigma_1) \in \mathcal{S}$ and remove it from $\mathcal{S}$.
    \item $(p_{S_1, \sigma_1}, v_{S_1,\sigma_1, 1},\dots, v_{S_1,\sigma_1, n})  =
\texttt{GeometricResolution}(f_i (i \in S_1), g, \sigma_1, d, \ell_\alpha)$.

 \item $({\rm Empty}, \tau_{S_1, \sigma_1, 1}, \dots,
\tau_{S_1, \sigma_1, k}) = \texttt{MinimumInGeometricResolution}(p_{S_1, \sigma_1}, v_{S_1,\sigma_1, 1},\dots, v_{S_1,\sigma_1, n},$
$f_1, \dots,$ $f_m, g, l, d)$.
\item While Empty = True:
\begin{enumerate}

\item Discard $(S_1, \sigma_1)$, take a new $(S_1, \sigma_1)$ and remove it from $\mathcal{S}$.
\item $(p_{S_1, \sigma_1}, v_{S_1,\sigma_1, 1},\dots, v_{S_1,\sigma_1, n})  =
\texttt{GeometricResolution}(f_i (i \in S_1), g,  \sigma_1, d, \ell_\alpha)$.

 \item $({\rm Empty}, \tau_{S_1, \sigma_1, 1}, \dots,
\tau_{S_1, \sigma_1, k}) = \texttt{MinimumInGeometricResolution}(p_{S_1, \sigma_1}, v_{S_1,\sigma_1, 1},\dots,$ $v_{S_1,\sigma_1, n},$
$f_1, \dots,$ $f_m, g, l, d)$.
\end{enumerate}

\item $\cI = \big\{\big((p_{S_1, \sigma_1}, v_{S_1,\sigma_1, 1},\dots, v_{S_1,\sigma_1, n}),
\tau_{S_1, \sigma_1, 1}\big),  \dots,
\big((p_{S_1, \sigma_1}, v_{S_1,\sigma_1, 1},\dots, v_{S_1,\sigma_1, n}), \tau_{S_1, \sigma_1, k}\big) \big\}$.

\item While $\mathcal{S}\ne \emptyset$:

\begin{enumerate}
 \item Take $(S_2, \sigma_2) \in \mathcal{S}$ and remove it from $\mathcal{S}$.
\item $(p_{S_2, \sigma_2}, v_{S_2,\sigma_2, 1},\dots, v_{S_2,\sigma_2, n})  =
\texttt{GeometricResolution}(f_i (i \in S_2), g, \sigma_2, d, \ell_\alpha )$.

 \item $({\rm Empty}, \tau_{S_2, \sigma_2, 1}, \dots,
\tau_{S_2, \sigma_2, k}) =
\texttt{MinimumInGeometricResolution}(p_{S_2, \sigma_2}, v_{S_2,\sigma_2, 1},\dots, v_{S_2,\sigma_2, n},$
$f_1, \dots,$ $f_m, g, l,d)$.
\item If Empty = False :

\begin{enumerate}
\item Sign = \texttt{ComparingMinimums}$(p_{S_1, \sigma_1},v_{S_1, \sigma_1, 1},\dots, v_{S_1, \sigma_1, n},
p_{S_2, \sigma_2},v_{S_2, \sigma_2, 1},\dots, v_{S_2, \sigma_2, n}, g, d$ $\tau_{S_1, \sigma_1, 1}, \tau_{S_2, \sigma_2, 1})$.

\item If Sign $ = 0$ then\\
$\cI = \cI \cup \big\{\big((p_{S_2, \sigma_2}, v_{S_2,\sigma_2, 1},\dots, v_{S_2,\sigma_2, n}),
\tau_{S_2, \sigma_2, 1}\big),$ $\dots,$
$\big((p_{S_2, \sigma_2}, v_{S_2,\sigma_2, 1},\dots, v_{S_2,\sigma_2, n}), \tau_{S_2, \sigma_2, k'}\big) \big\}$.

\item If Sign $ = 1$
\begin{enumerate}
\item $\cI = \big\{\big((p_{S_2, \sigma_2}, v_{S_2,\sigma_2, 1},\dots, v_{S_2,\sigma_2, n}),
\tau_{S_2, \sigma_2, 1}\big),$ $\dots,
\big((p_{S_2, \sigma_2}, v_{S_2,\sigma_2, 1},\dots, v_{S_2,\sigma_2, n}), \tau_{S_2, \sigma_2, k'}\big) \big\}$.

\item $(S_1, \sigma_1) = (S_2, \sigma_2)$.
\end{enumerate}

\end{enumerate}

\end{enumerate}

\end{enumerate}

Now we prove Theorem 1.

\begin{proof}{Proof of Theorem \ref{mainTheorem}.} The correctness of the algorithm follows from the results in Sections \ref{sec:deformation} and \ref{sec:geomres}.
The complexity upper bound is obtained by adding the complexity bounds for the subroutines \texttt{GeometricResolution} and \texttt{MinimumInGeometricResolution} applied to every element in $ \mathcal{S}$,  and \texttt{ComparingMinimums} applied successively to pairs of elements from $\mathcal{S}$.
\end{proof}

\end{document}